\documentclass{compositio}

\usepackage{mytemplate}

\newtheorem{thm}{Theorem}[section]
\newtheorem{lem}[thm]{Lemma}
\newtheorem{prop}[thm]{Proposition}

\newcommand{\f}{{\mathbf{f}}}
\newcommand{\g}{{\mathbf{g}}}
\renewcommand{\a}{{\mathbf{a}}}
\renewcommand{\b}{{\mathbf{b}}}

\newcommand{\Fr}{{\mathrm{Fr}}}
\renewcommand{\sl}{{\mathrm{sl}\,}}
\newcommand{\Disc}{{\mathrm{Disc}}}
\newcommand{\BDisc}{{\mathrm{BDisc}}}
\newcommand{\lf}{{\mathrm{lf}\,}}

\newcommand{\fqb}{{\overline{\F}_q}}
\newcommand{\fpb}{{\overline{\F}_p}}

\newcommand{\G}{P}
\newcommand{\FF}{\mathcal{F}}
\newcommand{\Gal}{\mathrm{Gal}}

\renewcommand{\k}{k}
\newcommand{\K}{{\mathcal{K}}}
\newcommand{\Kb}{{\overline{\K}}}
\newcommand{\D}{{\mathcal{D}}}

\begin{document}

\title{On the Bateman-Horn Conjecture for Polynomials over Large Finite Fields}
\author{Alexei Entin}
\email{aentin@stanford.edu}
\address{450 Serra Mall, Stanford, CA 94305, U.S.A.}
\thanks{The research leading to these results has received funding from the European Research Council under the European Union's Seventh Framework Programme (FP7/2007-2013) / ERC grant no. 320755.}

\begin{abstract} We prove an analogue of the classical Bateman-Horn conjecture on prime values of polynomials for the ring of polynomials over a large finite field.
Namely, given non-associate, irreducible, separable and monic (in the variable $x$) polynomials $F_1,\ldots,F_m\in\F_q[t][x]$, we show that the number of $f\in\F_q[t]$ of degree $n\ge\max(3,\deg_t F_1,\ldots,\deg_t F_m)$ such that all $F_i(t,f)\in\F_q[t],1\le i\le m$ are irreducible is $$\lb\prod_{i=1}^m\frac{\mu_i}{N_i}\rb q^{n+1}\lb 1+O_{m,\,\max\deg F_i,\,n}\lb q^{-1/2}\rb\rb,$$ where
$N_i=n\deg_xF_i$ is the generic degree of $F_i(t,f)$ for $\deg f=n$ and $\mu_i$ is the number of factors into which $F_i$ splits over $\fqb$.
Our proof relies on the classification of finite simple groups.

We will also prove the same result for non-associate, irreducible and separable (over $\F_q(t)$) polynomials $F_1,\ldots,F_m$ not necessarily monic in $x$ under the assumptions that $n$ is greater than the number of geometric points
of multiplicity greater than two on the (possibly reducible) affine plane curve $C$ defined by the equation $$\prod_{i=1}^mF_i(t,x)=0$$ (this number is always bounded above by $\lb\textstyle\sum_{i=1}^m\deg F_i\rb^2/2$, where $\deg$ denotes the total degree in $t,x$) and 
$$p=\mathrm{char}\,\F_q>\max_{1\le i\le m} N_i,$$ where $N_i$ is the generic degree
of $F_i(t,f)$ for $\deg f=n$.
\end{abstract}

\maketitle

\section{Introduction}

The classical Bateman-Horn conjecture \cite{bh} predicts the frequency at which a set of irreducible polynomials over the integers attains simultaneously prime values at integer points. Namely, let $F_1,\ldots,F_m\in\Z[x],\deg F_i>0$ be non-associate (i.e. no two differ just by a sign) irreducible polynomials over the integers and suppose that for each prime $p$
there exists $a\in\Z$ such that $p\nmid F_1(a)\cdots F_k(a)$. Then
$$\#\{1\le a<x | F_1(a),\ldots,F_k(a)\mbox{ are prime}\}\sim \frac{C(F_1,\ldots,F_m)}{\prod_{i=1}^m\deg F_i}\frac{x}{\log^m x},$$
where $$C(F_1,\ldots,F_m)=\prod_{p\,\mathrm{prime}}\frac{1-\nu(p)/p}{(1-1/p)^m},$$ $\nu(p)$ being the number of solutions to $F_1(x)\cdots F_m(x)\equiv 0
\pmod{p}$ in $\Z/p$. The only proved case of the conjecture is the case of a single linear polynomial, which is the Prime Number Theorem for arithmetic progressions.

In the present paper we establish an analogue of this conjecture for polynomials over large finite fields. Let $q$ be a power of a prime $p$, $\F_q$ the field
with $q$ elements. We will consider an analogue of the Bateman-Horn problem with the ring $\Z$ replaced by the one-variable polynomial ring $\F_q[t]$.
For polynomials in several variables we will denote by $\deg$ the total degree and by $\deg_t,\deg_x$, etc. the degree in the respective variable. 
Let $F_1,\ldots,F_m\in\F_q[t][x],\deg_x F_i>0$ be non-associate (i.e. not differing by a multiplicative constant in $\F_q^\times$), irreducible and separable over $\F_q(t)$, i.e. $F_i\not\in\F_q[t][x^p]$.
Let $n$ be a natural number. We ask for how many of the polynomials $f\in\F_q[t]$ with $\deg f=n$ all the values $F_i(t,f)\in\F_q[t]$ for $i=1,\ldots,m$ are irreducible. We are interested in the asymptotics of this quantity for fixed $m,\deg F_i,n$ and $q\to\ity$.
We will attack this problem by two different methods, each applicable under different additional conditions on $F_1,\ldots,F_m,n$ and 
$p=\mathrm{char}\,\F_q$, and obtain
two sets of results. The first method requires the classification of finite simple groups for its strongest form, while the second method is more direct and does not use any non-elementary facts from group theory.

Our first set of results applies to $F_i$ which are all monic in $x$. The second set of results, which applies also to the non-monic case, will be given as Theorem \ref{thmnm} at the end of the present section. To state our results we define the \emph{slope} of a polynomial $$P(t,x)=\sum_{j=0}^rc_j(t)x^j,c_r=1$$ which is monic in $x$ to be
\beq\label{slope}\sl P=\max_{1\le j\le r}\frac{\deg c_{r-j}(t)}{j}\eeq 
(the degree of the zero polynomial is $-\ity$).
The slope has the property that $\sl PQ\le\max(\sl P,\sl Q)$ and $\sl PQ=\sl P$ if $\sl P=\sl Q$ (to see the latter observe that if $\sl P=\sl Q,\deg_x P=r_P,\deg_x Q=r_Q$, and $j_P,j_Q$ are the largest indices for which the maximum in \rf{slope} is attained for $P$ and $Q$ respectively, then the degree of the coefficient of $x^{r_P+r_Q-j_P-j_Q}$ in $PQ$ is $(j_P+j_Q)\sl P$).
Also we always have $\sl P\le\deg_t P$.

Our main result for the monic case is the following.

\begin{thm}\label{main} Let $F_1,\ldots,F_m\in\F_q[t][x],\deg_x F_i=r_i>0$ be non-associate irreducible polynomials which are separable over $\F_q(t)$ $($i.e. $F_i\not\in\F_q[t][x^p])$ and monic in $x$.
Let $n$ be a natural number satisfying $n\ge 3$ and $n\ge\sl F_i$ for $1\le i\le m$. Denote $N_i=r_in$.
Denote by $\mu_i$ the number of irreducible factors into which $F_i(t,x)$ splits over $\fqb$.
Then
\begin{multline}\label{bh}\#\{f\in\F_q[t],\deg f=n | F(t,f)\in\F_q[t]\mbox{ is irreducible}\}=\\=\lb\prod_{i=1}^m\frac{\mu_i}{N_i}\rb q^{n+1}\lb 1+O_{m,\deg F_i,n}\lb q^{-1/2}\rb\rb,\end{multline}
the implicit constant in the $O$-notation depending only on $m,\deg F_i$ for $1\le i\le m$ and $n$.\end{thm}

The assumption $n\ge\sl F_i$ implies that $N_i=r_in$ is the generic degree of $F_i(t,f)$ for $\deg f=n$, i.e. if $a_0,\ldots,a_n$ are free variables
then $$\deg_t F_i\lb t,\sum_{j=0}^n a_jt^j\rb=N_i.$$

We note that the implied constant can be made explicit, but we do not concern ourselves with tracking it.
We conjecture that Theorem \ref{main}, as well as Theorems \ref{main_dec} and \ref{chowla} stated below, hold more generally without the monicity condition on the $F_i$, without any conditions on $n$ except $n\ge 3$ and for even $q$ as well.
The separability condition on the $F_i$ generally cannot be omitted, it is not difficult to construct inseparable polynomials violating the Bateman-Horn statistics or even not assuming any irreducible values on $\F_q[t]$. See \cite{CCG} and \cite{swan}.

The proof of Theorem \ref{main} requires the classification of finite simple groups. If we content ourselves with a result valid for 
$n\ge 3\deg_xF_i,1\le i\le m$ (as well as $n\ge\sl F_i$), then a much more elementary result from group theory is sufficient. See the discussion following
the statement of Theorem \ref{thmgroup}. The same applies to Theorem \ref{main_dec} below.

Several related, less general, results have been known previously. Bary-Soroker and Jarden \cite{jarden} established \rf{bh} for polynomials $F_1,\ldots,F_k$ which are characteristic-0-like and nodal (see \cite[\S 1]{jarden} for the precise definitions) in the case $n=1$. Bary-Soroker \cite{constant}
and Pollack \cite{pollack} treated the case of $F_i$ independent of $t$, i.e. $F_i\in\F_q[x]$. Bary-Soroker also treated the case 
$F_i=x+h_i,h_i\in\F_q[t],n\ge\deg h_i$ (analogue of the Hardy-Littlewood prime tuples conjecture) in \cite{hl}. The special case $m=2$ of the latter result was previously established by Bender and Pollack \cite{benderpollack}. Bank, Bary-Soroker and Rosenzweig \cite{arith} treated the case
of a single linear polynomial and $n\ge 3$ ($n\ge 2$ for odd $q$).

We follow the general strategy used in most of the cited work above, which reduces the Bateman-Horn conjecture to computing the Galois group of the set of polynomials $F_i(t,a_nt^n+\ldots+a_0)$ over the field $\F_q(a_0,\ldots,a_n)$, $a_0,\ldots,a_n$ being free variables, using a version of the Chebotarev density theorem (see the next section). The Galois group computation is the novel part of the present work. Unlike the previous results described above where the Galois group was
computed directly by algebraic means, we will use the arithmetic significance of the Galois group provided by the density theorem to prove a strong transitivity property, after which we will invoke results about multiply transitive groups. For the proof of Theorem \ref{thmnm} we will use a more direct algebraic method.

In the setting of Theorem \ref{main} we will not just compute the probability of all the $F_i(t,f)$ being irreducible but the probability of any possible decomposition. For simplicity we state here our result just for the case of absolutely irreducible polynomials, by which we mean polynomials irreducible over $\fqb$.

\begin{thm}\label{main_dec} Let $F_1,\ldots,F_m\in\F_q[t][x]$ and $n$ satisfy all the conditions of Theorem $\ref{main}$ and moreover assume that the $F_i$ are absolutely irreducible.
Denote $N_i=r_in$, where $r_i=\deg_x F_i$.

Fix partitions $$N_i=\sum_{k=1}^{M_i} s_{ik}, s_{ik}\ge 1$$ of each $N_i$ for $1\le i\le m$.
Then the number of $f\in\F_q[t],\deg f=n$ such that each $F_i(t,f)$ decomposes into $M_i$ irreducible factors of
degrees $s_{ik},1\le k\le M_i$ is 
$$\lb\prod_{i=1}^m\mathcal{P}\lb s_{i1},\ldots,s_{iM_i}\rb q^{n+1}\rb\lb 1+O_{k,\deg F_i,n}\lb q^{-1/2}\rb\rb,$$ 
where $\mathcal{P}\lb s_{i1},\ldots,s_{iM_i}\rb$ is the probability of
a random permutation in $S_{N_i}$ having the cycle structure $\lb s_{i1},\ldots,s_{iM_i}\rb$.\end{thm}

The connection between decompositions of polynomials and cycle structures of permutations will be made clear in the following section. Meanwhile note that for
absolutely irreducible $F_1,\ldots,F_m$, Theorem \ref{main_dec} implies Theorem \ref{main}, since the probability of a permutation in $S_{N_i}$ being a full (i.e. length $N_i$) cycle is
$1/N_i$. For not absolutely irreducible $F_i$ a similar result can be obtained with the usual permutation groups replaced by certain permutational wreath products. See Theorem \ref{thmgalois} in the next section.

As a byproduct we will also obtain the following result, which is a generalised Chowla conjecture for polynomials (the original Chowla conjecture for integers appears in \cite{chowla}):

\begin{thm}\label{chowla} Let $F_1,\ldots,F_m$ and $n$ satisfy the conditions of Theorem $\ref{main}$, except the condition $n\ge 3$ is not required if $q$ is odd (if $q$ is even we still require $n\ge 3$). Then for every sequence $s_1,\ldots,s_m=\pm 1$ of signs, the number of $f\in\F_q[t],\deg f=n$ such that $\mu(F_i(t,f))=s_i$ for $1\le i\le m$ is $$\frac{q^{n+1}}{2^m}\lb 1+O_{k,\deg F_i,n}\lb q^{-1/2}\rb\rb.$$
Here $\mu$ denotes the M\"{o}bius function for $\F_q[t]$.\end{thm}

It is easy to see that in the case of absolutely irreducible polynomials Theorem \ref{chowla} follows from Theorem \ref{main_dec}, but in fact if $q$ is odd it follows directly from the much simpler Proposition \ref{propsign} on the multiplicative independence of the discriminants $\Disc_tF_i(t,a_0+\ldots+a_nt^n)\in\fqb(a_0,\ldots,a_n)$ ($a_0,\ldots,a_n$ being free variables) modulo squares in $\fqb(a_0,\ldots,a_n)^\times$. Similarly, in the even characteristic case it follows from Proposition 5.2, which is a similar statement about the linear independence of Berlekamp discriminants. The proof of Theorem \ref{chowla} or Propositions \ref{propsign}, 5.2 does not require the classification of finite simple groups or any other non-elementary fact from group theory. The equivalence of Theorem \ref{chowla} and Propositions \ref{propsign}, 5.2 is shown by the method of Carmon and Rudnick \cite{carmonrudnick}, which they applied to the special case of Theorem 
\ref{chowla} with $F_i=x+h_i$ for $h_i\in\F_q[t]$ with $n\ge\deg h_i$ and $q$ odd. Carmon recently generalised the result to even characteristic \cite{carmon}. 
The only part of their proof which requires this special form is proving Proposition \ref{propsign} for this case.
In fact the only part of the present work that requires the conditions that all the $F_i$ are monic in $x$ and that $n\ge\sl F_i$ for all $i$ is in the proof of Propositions \ref{propsign}, 5.2, so if one can prove Theorem \ref{chowla} or Propositions \ref{propsign}, 5.2 without these conditions then one can dispense with them in all of our results. We conjecture that for $n\ge 3$ these conditions are not required.

Next we state our second set of results obtained by a different method, which apply to not necessarily monic $F_1,\ldots,F_m$. We will need to consider the
possibly reducible affine plane curve $C$ over $\F_q$ defined by the equation $F(t,x)=0$, where $F=\prod_{i=1}^mF_i$. For a point $P\in C(\fqb)$ we denote by
$m_P$ its multiplicity in $C$, i.e. the degree of the lowest-degree form appearing in the Taylor expansion of $F(t,x)$ around $P$.

\begin{thm}\label{thmnm} Let $F_1,\ldots,F_m\in\F_q[t][x]$ with $\deg_x F_i>0$ be non-associate irreducible polynomials which are separable over $\F_q(t)$, i.e. $F_i\not\in\F_q[t][x^p]$. Set $F=\prod_{i=1}^m F_i$. Let $C/\fqb$ be the $($possibly reducible$)$ affine plane curve defined by $F(t,x)=0$. Let $n\ge 3$ be a natural number satisfying
$$n>\#\{P\in C(\fqb)|m_P>2\}.$$
Denote $N_i=\deg_t F_i(t, a_nt^n+\ldots+a_0)$ where $a_0,\ldots,a_n$ are free variables.
Denote by $\mu_i$ the number of irreducible factors into which $F_i(t,x)$ splits over $\fqb$.
Assume that $$p>\max_{1\le i\le m} N_i.$$
Then the assertions of Theorem $\ref{main}$, Theorem $\ref{main_dec}$ and Theorem $\ref{chowla}$ hold in this case as well.
\end{thm}

The proof of Theorem \ref{thmnm} is also based on a Galois group computation, like the proof of theorems \ref{main}, \ref{main_dec} and \ref{chowla}, but the computation is more direct and does not use any non-elementary results from group theory. It will be carried out in sections \ref{secnm} and \ref{seckeyprop}.

\section{Galois groups, Frobenius classes and equidistribution}\label{secequidist}

Let $p$ be a prime number, $q$ a power of $p$.
Let $g\in\F_q[t]$ be a separable polynomial (i.e. having distinct roots over $\fqb$) of degree $N$. The Frobenius map $\Fr_q$ (given by $y\mapsto y^q$) defines a permutation of the
roots of $g$, which gives a well-defined conjugacy class $\Th(g)$ of the symmetric group $S_N$. The degrees of the prime factors of $g$ correspond to the
cycle lengths of $\Th(g)$. In particular $g$ is irreducible iff $\Th(g)$ is (the conjugacy class of) a full cycle. We call $\Th(g)$ the Frobenius class of $g$. If $g_1,\ldots,g_m\in\F_q[t]$ are separable polynomials with $\deg g_i=N_i$ we get a conjugacy class $\Th(g_1,\ldots,g_m)$ in $S_{N_1}\times\ldots\times S_{N_m}$ by taking the product of the individual Frobenius classes $\Th(g_i)$. We call $\Th(g_1,\ldots,g_m)$ the Frobenius class of $g_1,\ldots,g_m$.

Let $F_1,\ldots,F_m\in\F_q[t][x],\deg_x F_i>0$ be non-associate, irreducible and separable over $\F_q(t)$, $n$ a natural number. Set $F=\prod_{i=1}^m F_i$. 
Note that by our assumptions $F$ is separable over $\F_q(t)$. Let $a_0,\ldots,a_n$ be free variables, 
$\f=a_nt^n+\ldots+a_0\in\F_q[\a,t]$ ($\a$ is a shorthand for $a_0,\ldots,a_n$) and
$N_i=\deg_t F_i(t,\f)$. 
\\ \\
{\bf Convention.} For the rest of the paper the asymptotic big-$O$ notation always implies a constant depending only on $m$, $\deg F_i$ for $1\le i\le m$ and $n$.
\\

To proceed further we need the following lemma:

\begin{lem}\label{lemsep}
Under the above assumptions on the $F_i$ and $F=\prod F_i$ the polynomial $F(t,\f)\in \F_q[\a][t]$ is separable over $\F_q(\a)$.
For all but $O(q^n)$ of the polynomials $f\in\F_q[t]$ such that $\deg f=n$ $($the total number of such polynomials is $q^n(q-1))$, the polynomials $F_i(t,f)\in\F_q[t]$ have degree $N_i$ and $F(t,f)\in\F_q[t]$ is separable.
\end{lem}

\begin{proof} This is proved in \cite{sqfree}.\end{proof}

The second part of the lemma implies that for all but $O(q^n)$ of the polynomials $f\in\F_q[t],\deg f=n$ (again note that the total number of such polynomials
is $q^n(q-1)$) the Frobenius class
$\Th(F_1(t,f),\ldots,F_m(t,f))$ is a well-defined conjugacy class in $S_{N_1}\times\ldots\times S_{N_m}$.

Let $L$ be the splitting field of $F(t,\f)$ over $\F_q(\a)$. Denote by $G$ its Galois group. It can be viewed as a subgroup of
$S_{N_1}\times...\times S_{N_m}$ by its action on the roots of each $F_i(t,\f)$. These roots are all distinct by the first assertion of Lemma \ref{lemsep}.
Let $\F_{q^\nu}$ be the algebraic closure of $\F_q$ in $L$. Denote by $G_1$ the set of $\sig\in G$ such that $\sig$ acts as the Frobenius map $\Fr_q$ on 
$\F_{q^\nu}$. It is a coset of the normal subgroup $\Gal(L/\F_{q^\nu}(\a))\ss G$.

The fundamental tool we will use in the present work is the following equidistribution result:

\begin{thm}\label{equidist} Let $F_1,\ldots,F_m\in\F_q[t][x],\deg_x F_i>0$ be non-associate, irreducible and separable over $\F_q(t)$, $F=\prod_{i=1}^m F_i$.
Here we do not assume that the $F_i$ are monic in $x$.
Let $n$ be a natural number, $a_0,\ldots,a_n$ be free variables, $\a=(a_0,\ldots,a_n)$, $\f=\sum_{j=0}^n a_jt^j$, $N_i=\deg_t F_i(t,\f)$.
Denote by $L$ the splitting field of $F(t,\f)$ over $\F_q(\a)$, $G=\Gal(L/\F_q(\a))$ its Galois group, $\F_{q^\nu}$ the algebraic closure of $\F_q$ in $L$.
Denote by $G_1$ the set of $\sig\in G$ acting as $\Fr_q$ on $\F_{q^\nu}$. Consider $G$ as a subgroup of $S_{N_1}\times...\times S_{N_m}$ via its action on the roots of $F_i(t,\f)$.

Then for every conjugacy class $C$ in $S_{N_1}\times...\times S_{N_m}$ we have
\begin{multline*}\#\{f\in\F_q[t],\deg f=n|\Th(F_1(t,f),\ldots,F_m(t,f))=C\}=\\ \frac{\#(C\cap G_1)}{\#G_1}q^{n+1}\lb 1+O_{m,\deg F_i,n}(q^{-1/2})\rb.$$
\end{multline*}
\end{thm}

This result is now quite standard and can be seen for example as a special case of \cite[Proposition 2.2]{constant}. A variant of Theorem \ref{equidist} appears as \cite[Theorem 3.1]{andrade}. There only 
the case of monic $F_i$ and $\nu=1$ is considered. However the result is deduced from a more general explicit Chebotarev theorem \cite[Theorem A.4]{andrade}, which implies Theorem \ref{equidist} in the same way.
Theorem \ref{equidist} can also be viewed as a 0-dimensional case of Deligne-Katz equidistribution \cite{KS}, with the Galois group acting as the monodromy group. Our results (especially Theorem \ref{main_dec}) can then be seen as a 0-dimensional disconnected fiber analogue of the much deeper equidistribution results of \cite{KS}.

Theorem \ref{equidist} reduces the study of the factorization statistics of $$F_1(t,f),\ldots,F_m(t,f)$$ to the computation of the Galois group
$G=\Gal(L/\F_q(\a))$ as a permutation group on the roots of $F_i(t,\f)$ over $\F_q(\a)$. It is this computation which is the heart of the present work.
One of the novelties in our work is that unlike all the previous work cited in the introduction we actually use Theorem \ref{equidist} in the computation of the Galois group (in the proof of multiple transitivity, see the next section) and not only apply it after the Galois group has been computed directly.

Theorem \ref{main_dec} would follow from Theorem \ref{equidist} if we can show that for $F_1,\ldots,F_m$ satisfying the conditions of Theorem \ref{main_dec}
the Galois group of $F(t,\f)$ ($F=\prod F_i$) over $\F_q(\a)$ is the maximal possible, i.e. $S_{N_1}\times...\times S_{N_m}$. In the general case (not necessarily absolutely irreducible $F_i$) the answer is a little bit more complicated and will be stated next.

Let $\G\in\F_q[t][x]$ be irreducible and separable. There is some minimal field $\F_{q^\mu}$ over which $\G$ splits into absolutely irreducible factors.
It is not difficult to see that the number of these factors is $\mu$ and they are transitively permuted by the Galois group $\Gal(\F_{q^\mu}/\F_q)$ which is cyclic
of order $\mu$.

\begin{thm}\label{thmgalois} Let $F_1,\ldots,F_m\in\F_q[t][x],\deg_x F_i=r_i>0$ and $n\ge 3$ satisfy the conditions of Theorem $\ref{main}$.
Let $a_0,\ldots,a_n$ be free variables, $\a=(a_0,\ldots,a_n)$, $\f=\sum_{j=0}^n a_jt^j$, $N_i=\deg_t F_i(t,\f)=r_in$ $($the last equality holds because
$n\ge\sl F_i)$.

Let $$F_i=\prod_{j=1}^{\mu_i}\G_{ij}, \G_{ij}\in\F_{q^{\mu_i}}[t][x]$$ be the decomposition of $F_i$ into absolutely irreducible factors.
It is easy to see that $\deg_t \G_{ij}(t,\f)=N_i/\mu_i$. Let $\F_{q^\mu}$ be the composite of all the $\F_{q^{\mu_i}}$.

Denote by $L$ the splitting field of $F(t,\f)$ over $\F_q(\a)$ and let $G=\Gal(L/\F_q(\a))$ be its Galois group, which can be considered a subgoup of 
$S_{N_1}\times...\times S_{N_m}$ through its action on the roots of $F_i(t,\f)$. Denote by $\Om_{ij}$ the set of roots of $\G_{ij}$ in $L$.
The Galois group $H_\mu=\Gal(\F_{q^\mu}/\F_q)=\Gal(\F_{q^\mu}(\a)/\F_q(\a))$ is isomorphic to $\Z/\mu$ and for each $i$ it acts on the set $\{\G_{ij}\}_{1\le j\le\mu_i}$ transitively, the action
factors through a principal action of $H_{\mu_i}=\Gal(\F_{q^{\mu_i}}/\F_q)\cong \Z/\mu_i$. For $\sig\in H_\mu$ we denote by $\Om_{ij}^\sig$ the set of roots of $\G_{ij}^\sig$. The Galois action of $\sig$ on the roots
sends $\Om_{ij}$ into $\Om_{ij}^\sig$.

The following holds: $G$ consists of all the permutations $\pi$ on $\bigcup_{ij}\Om_{ij}$ for which there exists some $\sig\in H_\mu$ such that 
$\pi\lb\Om_{ij}\rb=\Om_{ij}^\sig$ for each $i$ and $j$.
\end{thm}

Note that if all the $F_i$ are absolutely irreducible, i.e. $\mu_i=1$ for $1\le i\le m$, then we get $G=S_{N_1}\times...\times S_{N_m}$. Theorem \ref{maingalois} gives a complete description of the Galois group of $\F_i(t,\f)$ over $\F_q(\a)$ as a permutation group together with its map to $H_\mu=\Gal(\F_{q^\mu}/\F_q)\cong\Z/\mu$. It is now an elementary exercise on permutation groups to deduce Theorems \ref{main}, \ref{main_dec} and \ref{chowla} from Theorem \ref{equidist} and Theorem \ref{thmgalois}.
For example Theorem \ref{main} follows from the fact that the probability of a random permutation $\pi$ of $\bigcup_{ij}\Om_{ij}$ satisfying $\pi\lb\Om_{ij}\rb=\Om_{ij}^\sig$ for all $i,j$ for a given generator $\sig$ of $H_\mu$ (considered as a cyclic permutation on $\{1,\ldots,\mu_i\}$ for each $i$) being a full cycle on each $\bigcup_{j=1}^{\mu_i}\Om_{ij}$ is $\prod_{i=1}^m(\mu_i/N_i)$. The elementary proof of this fact is carried out in \cite{constant}. Theorem \ref{main_dec} follows at once from Theorem \ref{equidist}, Theorem \ref{thmgalois} and the remark immediately following the statement of Theorem \ref{thmgalois}.

The essential result from which Theorem \ref{thmgalois} will follow is the following.

\begin{thm}\label{main_galois} Let $p>2$ be a prime number and $\k=\fpb$. Let $$F_1,\ldots,F_m\in \k[t][x],\deg_x F_i=r_i>0$$ be non-associate, irreducible, monic in $x$ and separable over $\fpb(t)$. Let $n\ge 3$ be a natural number such that $n\ge\sl F_i$ for $1\le i\le m$. Let $a_0,\ldots,a_n$ be free variables, $\a=(a_0,\ldots,a_n)$, $\f=\sum_{j=0}^n a_jt^j$, $N_i=\deg_t F_i(t,\f)=r_in$.
Denote by $L$ the splitting field of $\prod_{i=1}^mF(t,\f)$ over $k(\a)$, $G=\Gal(L/\k(\a))$ its Galois group. Then $G$ is the full permutation group
$S_{N_1}\times...\times S_{N_m}$ acting on the roots of $F_i(t,\f)$ over $\k(\a)$.\end{thm}

The proof of Theorem \ref{main_galois} will occupy the next three sections and is the heart of the present work. Theorem \ref{thm_galois} follows directly from Theorem \ref{main_galois} applied to the factors of the $F_i$ over $\F_{q^\mu}$ (we use the notation of Theorem \ref{thmgalois}). To see that it can be applied to the factors it needs to be verified that $n\ge\sl F_i$ implies the same for the factors of $F_i$ over $\F_{q^\mu}$. But these factors have the same slope as $F_i$ since $\sl PQ=\sl P$ whenever $\sl P=\sl Q$. So Theorem 
\ref{main_galois} applies to the factors of the $F_i$ over $\F_{q^\mu}$.
Now in the notation of Theorem \ref{thmgalois} every $\sig\in H_\mu=\Gal(\F_{q^\mu}/\F_q)=\Gal(\F_{q^\mu}(\a)/\F_q(\a))$ lifts
to some $\sig'\in\Gal(L/\F_q(\a))$. It can then be composed with some element of $\Gal(L/\F_{q^\mu})$ (which can be chosen to permute each $\Om_{ij}$ as we please) to obtain any permutation of the form described in the assertion of Theorem \ref{thmgalois}.

Thus Theorem \ref{main_galois} implies Theorem \ref{thmgalois}, which in turn implies theorems \ref{main}, \ref{main_dec} and \ref{chowla}. By the same considerations Theorem \ref{thmnm} follows from the following
 
\begin{thm}\label{maingalois} Let $p$ be prime. Let $F_1,\ldots,F_m\in\fpb[t][x],\deg_x F_i>0$ be non-associate irreducible polynomials which are separable over $\fpb(t)$ and not necessarily monic in $x$. Denote $F=\prod_{i=1}^m F_i$. Let $C$ be the affine plane curve defined by $F(t,x)=0$.
Let $n\ge 3$ be a natural number such that $$n>\#\{P\in C(\fpb)|m_P>2\}$$ (the definition of the multiplicity $m_P$ is given just before the statement of Theorem \ref{thmnm}). 
Denote $\f=\sum_{j=0}^na_it^i\in\fpb(a_0,\ldots,a_n)[t]$, where $\a=(a_0,\ldots,a_n)$ are free variables and $N_i=\deg_t F_i(t,\f)$. 
Assume that $$p>\max_{1\le i\le m} N_i.$$ Let $L$ be the splitting field of $F(t,\f)$ over $\fpb(\a)$ and let $G=\Gal(L/\fpb(\a))$ be its Galois group, which we view as a subgroup of
$S_{N_1}\times...\times S_{N_m}$ via its action on the roots of $F_i(t,\f)$. Then in fact
$G=S_{N_1}\times...\times S_{N_m}$ is the full product of permutation groups.\end{thm}

The proof of Theorem \ref{maingalois} will be given in Section \ref{secnm}.
\\ \\
{\bf Remark.} Theorems \ref{main_galois} and \ref{maingalois} hold in fact for any algebraically closed field $k$ and not just $k=\fpb$, provided the required conditions on the characteristic are satisfied or the characteristic is 0. This is because for any fixed
$m,n,\deg F_i$ the assertion can be formulated in the first-order language of fields and all the algebraically closed fields with a given characteristic
are elementarily equivalent for this language. Furthermore if a statement in the first-order language of fields holds in algebraically closed fields of
arbitrarily large characteristic it must also hold in characteristic 0. See \cite[\S 3.2]{marker}.

\section{Computing the Galois group - an outline}\label{secoutline}

In the present section we outline the proof of Theorem \ref{main_galois}, from which all our other results follow (see Section \ref{secequidist}). We will show that under the assumptions of Theorem \ref{main_galois} for each $F_i$
the Galois group $G_i$ of each $F_i(t,\f)$ over $k(\a)$ is $S_{N_i}$. Furthermore we will show that the permutation sign map $G\to\{\pm 1\}^m$ is onto. This will
finish the proof of Theorem \ref{main_galois} by the following elementary lemma on permutation groups.

\begin{lem}\label{lem1} Let $G\ss S_{N_1}\times...\times S_{N_m}$ be such that the projections $G\to S_{N_i}$ are onto and the permutation sign map 
$$G\to\prod_{i=1}^mS_{N_i}/A_{N_i}\cong\{\pm 1\}^m$$ is also onto.
Then $G=S_{N_1}\times...\times S_{N_m}$.\end{lem}

\begin{proof} See \cite[Lemma 3.2]{constant}.\end{proof}

In Section \ref{sectrans} we will prove that each $G_i$ is $(n+1)$-transitive (as a permutation group on the roots of $\F_i(t,\f)$ (Proposition \ref{trans}). Since we always assume $n\ge 3$ this implies that each $G_i$ is 4-transitive. While the result is stated
over $\fpb$, our proof will use the arithmetic significance of the Galois group, namely we will use Theorem \ref{equidist} in this step as well.

In Section \ref{secsign} we will show that the sign map $G\to\{\pm 1\}^m$ is onto. We call this the sign-independence property. In particular $G_i\not\ss A_{N_i}$ for each $i$. Next we will use the following deep fact from group theory (see \cite[Theorem 4.11]{cameron} or \cite[\S 7.3]{dixon}):

\begin{thm}\label{thmgroup} Let $G\ss S_N$ be a $4$-transitive permutation group not contained in $A_N$. Then $G=S_N$.\end{thm}

In fact it is known that except for $S_N$ and $A_N$ the only 4-transitive groups are simple Mathieu groups (which then must be contained in $A_N$). The proof of this fact requires the classification of finite simple groups, more precisely the Schreier conjecture (that the outer automorphism group of each finite simple group is solvable) which follows from it.

Theorem \ref{thmgroup} combined with the 4-transitivity of the $G_i$ and the fact that $G_i\not\ss A_{N_i}$ implies that $G_i=S_{N_i}$. Combined with the sign-independence property of $G$ and Lemma \ref{lem1} this shows that $G=S_{N_1}\times...\times S_{N_m}$ as asserted in Theorem \ref{main_galois}.

{\bf Remark.} By much more elementary means it can be shown that a $\ceil{3\sqrt{N}-2}$-transitive group $G\ss S_N$ is either $S_N$ or $A_N$, see \cite[\S 5.7]{hall}. Therefore if in addition to the assumptions of Theorem \ref{main} we assume $n\ge 3\deg_x F_i,1\le i\le m$, our results can be proved without using the classification of finite simple groups.

\section{Computing the Galois group - multiple transitivity}\label{sectrans}

In the present section we will prove the $(n+1)$-transitivity property for the Galois groups of the individual $F_i(t,\f)$ asserted after the statement of Lemma \ref{lem1}. Let $p$ be a prime number. Let $F_1\in\fpb[t][x],\deg_x F_1>0$ be irreducible and separable over $\fpb(t)$. In the present section we do not assume that $F_1$ is monic in $x$, nor that $p$ is odd. Neither do we require any conditions on $n$. Proposition \ref{trans} below is valid in this generality.

Let $n$ be a natural number, $\a=(a_0,\ldots,a_n)$ free variables and $f=\sum_{j=0}^na_it^i\in\fpb[\a][t]$. Denote $N=N_1=\deg_tF_1(t,\f)$.
Let $\al_1,\ldots,\al_N$ be the roots of $F_1(t,\f)$ in the algebraic closure of $\fpb(\a)$. By Lemma \ref{lemsep} they are distinct. Let 
$G=G(\fpb(\a,\al_1,\ldots,\al_n)/\fpb(\a))$ be the Galois group of $F_1(t,\f)$. We view $G$ as a permutation group on $\al_1,\ldots,\al_N$.

In the present section we prove the following transitivity property:

\begin{prop}\label{trans} The action of $G$ on $\al_1,\ldots,\al_N$ is $(n+1)$-transitive, i.e. every sequence of $n+1$ distinct roots of $F_1(t,\f)$ can be mapped to any other such sequence by some element of $G$.\end{prop}

Although the proposition is formulated over an algebraically closed field and one might expect a purely algebraic or algebro-geometric proof for it, our
approach is actually to use its arithmetic significance implied by Theorem \ref{equidist}.

Denote by $L$ the smallest extension of $\F_p(\a)$ containing the coefficients of $F_1$ and $\al_1,\ldots,\al_N$. It is finitely generated over $\F_p$ and so the
algebraic closure of $\F_p$ in $L$ is a finite field $\F_q$. Replacing $q$ by a large enough power and replacing $L$ by $L\F_q$ we may assume that 
$\Gal(L/\F_q(\a))=\Gal(L\fqb/\fqb(\a))=G$. The field $L$ is the splitting field of $F_1(t,\f)$ over $\F_q(\a)$ and by our assumptions $\F_q$ is algebraically closed in $L$. These properties persist if we replace $q$ by any power of it. This will be used later.

In the present section we continue using the convention of section \ref{secequidist} that the asymptotic $O$-notation has an implied constant depending on $n,\deg F_1$.
For a polynomial $g\in\F_q[t]$ and a natural number $e$ we denote by $\ell_e(g)$ the number of length-$e$ sequences of distinct roots of $g$ in the coefficient field $\F_q$.
We have of course $\ell_e(g)=\prod_{i=0}^{e-1}(\ell_1(g)-i)$. For a permutation $\sig\in S_N$ we will also denote by $\ell_e(\sig)$ the number of length-$e$ sequences of distinct fixed points of $\sig$. This is well defined on conjugacy classes in $S_N$.
Observe that for a separable $g$ with $\deg g=N$ we have 
$\ell_e(g)=\ell_e(\Th(g))$ where $\Th(g)$ is the Frobenius class of $g$. In the present section the Frobenius classes are defined via the action of $\Fr_q$
(not $\Fr_p$).

\begin{prop}\label{keycount} We have $$\sum_{f\in\F_q[t]\atop{\deg f=n}}\ell_{n+1}(F_1(t,f))=q^{n+1}\lb 1+O(q^{-1/2})\rb.$$\end{prop}

\begin{proof} 
Denote by $C$ the affine plane curve defined by $F_1(t,x)=0$. It is absolutely irreducible since $F_1$ is absolutely irreducible. 
Let $C^{n+1}$ be the $n+1$-fold product of $C$ with itself and $V\ss C^{n+1}$ the open subset of $(n+1)$-tuples of points with distinct $t$-coordinates,
$X=C^{n+1}\sm V$ its closed complement. The variety $C^{n+1}$ is irreducible and defined by equations of degree $O(1)$. The proper subvariety $X$ is also defined by equations of degree $O(1)$. Therefore by the Lang-Weil estimates we have $\#V(\F_q)=q^{n+1}\lb 1+O(q^{-1/2})\rb$.
For every sequence of points $(\tau_i,\xi_i)\in C(\F_q),i=1,\ldots,n+1$ with distinct $\tau_i$ there is a unique polynomial $f\in\F_q[t],\deg f\le n$ such that
$f(\tau_i)=\xi_i,1\le i\le n+1.$

Now by the definition of $\ell_e$ we have

\begin{multline}\label{doublecount}\sum_{f\in\F_q[t]\atop{\deg f\le n}}\ell_{n+1}(F_1(t,f))=\\=
\sum_{\tau_1,\ldots,\tau_{n+1}\in\F_q\atop{\mathrm{distinct}}}\#\{f\in\F_q[t],\deg f\le n,F_1(\tau_i,f(\tau_i))=0\}=\\=
\sum_{(\tau_1,\xi_1),\ldots,(\tau_{n+1},\xi_{n+1})\in C(\F_q)\atop{\tau_i\,\mathrm{distinct}}}\#\{f\in\F_q[t],\deg f\le n,f(\tau_i)=\xi_i\}=
\\=\#V(\F_q)=q^{n+1}\lb 1+O(q^{-1/2})\rb. \end{multline}
Since the number of $f$ with $\deg f<n$ is $q^n$ we may replace the condition $\deg f\le n$ in the summation with $\deg f=n$, introducing an error of
$O(q^n)$.
\end{proof}
   
Recall that $G=\Gal(L/\F_q(\a))\ss S_N$ and $\F_q$ is algebraically closed in $L$. For a random variable $X$ on a finite probability space $S$ we will denote
by $\av{X(s)}_{s\in S}$ its expected value. We always assume the probability measure to be uniform on the space.

\begin{prop}\label{transcrit} We have $$\av{\ell_{n+1}(\sig)}_{\sig\in G}=1.$$\end{prop}

\begin{proof} By Theorem \ref{equidist} the Frobenius elements of $F_1(t,f)$ for $f\in\F_q[t],\deg f=n$ are equidistributed in the $S_N$-conjugacy classes
of $G$ up to $O(q^{-1/2})$ (note that $\nu=1$ in the notation of Theorem \ref{equidist} since $\F_q$ is algebraically closed in $L$). Using Proposition \ref{keycount}
we see that \begin{multline*}\av{\ell_{n+1}(\sig)}_{\sig\in G}=\av{\ell_{n+1}(\Th(F_1(t,f)))}_{f\in\F_q[t],\deg f=n}+O(q^{-1/2})=\\=1+O(q^{-1/2})
\end{multline*}
(we may disregard those $f$ with $F_1(t,f)$ non-separable by Lemma \ref{lemsep}).
The implicit constant in the error term depends only on $\deg F,n$ and not on $q$.
We have observed in the beginning of the section that $q$ may be replaced by any power of $q$ with all of our assumptions remaining valid. Since $\av{\ell_e(\sig)}_{\sig\in G}$ is a rational number with denominator
dividing $N!$, replacing $q$ with a large enough power of it we see that we must have an equality $\av{\ell_{n+1}(\sig)}_{\sig\in G}=1$.
\end{proof}

To complete the proof of Proposition \ref{trans} we need the following elementary lemma from group theory.

\begin{lem}\label{lemtrans} Let $G$ be a finite group acting on a finite set $X$. For $\sig\in G$ denote by $\ell_e(\sig)$ the number of length-$e$ sequences of distinct fixed points for the action of $G$ on $X$. 
Then $$\av{\ell_e(\sig)}_{\sig\in G}\ge 1,$$ with equality iff $G$ is $e$-transitive.\end{lem}

\begin{proof} First we prove the assertion for $e=1$. For $x\in X$ we denote by $O_x$ its orbit and by $G_x$ its stabilizer. We have
\begin{multline*}\frac{1}{\# G}\sum_{\sig\in G}\ell_1(\sig)=\frac{1}{\# G}\#\{(\sig,x)\in G\times X|\sig x=x\}=\\=\frac{1}{\#G}\sum_{x\in X}\# G_x=
\sum_{x\in X}\frac{1}{\#O_x}\ge 1,\end{multline*}
since $\#O_x\le \# X$ for all $x$. Equality holds iff $O_x=X$ for all $x\in X$, i.e. if the action of $G$ is transitive.

To prove the assertion for general $e$ consider the set $X^{(e)}$ of $e$-sequences of distinct elements in $X$ with the $G$-action defined by
$\sig(x_1,\ldots,x_e)=(\sig x_1,\ldots,\sig x_e)$. For $\sig\in G$ the number of fixed points for this action is exactly $\ell_e(\sig)$. The action of $G$ on 
$X^{(e)}$ is transitive iff the action of $G$ on $X$ is $e$-transitive. Now applying the case $e=1$ to the action of $G$ on $X^{(e)}$ we obtain our assertion.
\end{proof}

Combining Proposition \ref{transcrit} and Lemma \ref{lemtrans} we see that the action of $G$ on $\al_1,\ldots,\al_N$ is $(n+1)$-transitive, which finishes the proof of Proposition \ref{trans}.

\section{Computing the Galois group - sign independence}\label{secsign}

Let $p$ be a prime number. Denote $\k=\fpb$. Let $F_1,\ldots,F_m\in\k[t][x],\deg_x F_i=r_i>0$ be non-associate, irreducible and separable over $\k(t)$. 
Denote $F=\prod_{i=1}^m F_i$. Let $n$ be a natural number, $\a=(a_0,\ldots,a_n)$ free variables, $f=\sum_{j=1}^na_it^i\in\k[\a][t]$. Assume that $n\ge\sl F_i,1\le i\le m$. Denote $N_i=r_in=\deg_tF_i(t,\f)$ (the last equality follows from the assumption $n\ge\sl F_i$).
Let $L$ be the splitting field of $F(t,\f)$ over $k(\a)$, $G=\Gal(L/k(\a))$ its Galois group, which we view as a subgroup of $S_{N_1}\times...\times S_{N_m}$
via its action on the roots of $F_i(t,\f)$.

From Section \ref{secoutline} we know that to complete the proof of Theorem
\ref{main_galois} it is enough to show that the sign projection map $G\to\prod_{i=1}^mS_{N_i}/A_{N_i}\cong\{\pm 1\}^m$ is onto. In the case of odd $p$, by a well-known fact from Galois theory this is equivalent
to the discriminants $\Disc_t F_i(t,\f)$ being linearly independent as elements of $\k(\a)^\times/\k(\a)^{\times 2}$ (note that they are nonzero by Lemma
\ref{lemsep}). This fact follows from the expression of the discriminant of a polynomial $g=\sum_{i=0}^N b_it^i,b_N\neq 0$ over an arbitrary field $K$ as
$\Disc\,g=\Del^2$ where $$\Del=b_N^{N-1}\prod_{i<j}(\rho_i-\rho_j),$$ $\rho_i$ being the roots of $g$. The expression $\Del$ is fixed by even permutations but not by odd ones (in odd characteristic), so $\Del\in K$ iff the Galois group of $g$ over $K$ is contained in $A_N$ (this argument applies to a single polynomial, but easy to extend to the case of several $g_1,\ldots,g_k$). This is valid only in odd characteristic. A similar criterion can be formulated in characteristic 2 using Berlekamp discriminants \cite{berlekamp}.

\subsection{Sign independence: odd p}

Assume $p>2$. We will need the following basic facts about the discriminant (see \cite[\S 12]{gelfand}). For every natural number $N$ there is a universal polynomial $\D_N(b_0,\ldots,b_N)\in\Z[b_0,\ldots,b_N]$ such that
over any field $K$ and for any $g=\sum_{j=0}^NB_jt^j\in K[t],B_N\neq 0$ we have $$\Disc\,g=\D_N(B_0,\ldots,B_N).$$ Furthermore, if we assign
to each variable $b_j$ the weight $j$, the polynomial $\D_N(b_0,\ldots,b_N)$ is homogeneous of (weighted) degree $N(N-1)$.

For polynomials $g_1,g_2$ with nonzero discriminants we have $$\Disc\,g_1g_2\equiv\Disc\,g_1\Disc\,g_2\bmod K^{\times2}.$$
Therefore to show the multiplicative independence of the discriminants $\Disc_t F_i(t,\f)$ (modulo squares) it is enough to show that the discriminant of any partial
product of the $F_i$ is not a square. Without loss of generality we may assume that this partial product is $F=\prod_{i=1}^m F_i$ (otherwise repeat the argument
with a subset of the $F_i$). Note that since $\sl PQ\le\max(\sl P,\sl Q)$ we have $n\ge\sl F$. So it is enough to prove the following.

\begin{prop}\label{propsign} Let $F\in\k[t][x]$ with $\deg_x F=r>0$ be separable over $\k(t)$ and monic in $x$. Assume that $n\ge\sl F$ and as usual $\a=(a_0,\ldots,a_n)$ are free variables,
$\f=\sum_{j=0}^na_jt^j$. Then $\Disc_tF(t,\f)$ is not a square in $\k(\a)$.\end{prop}

\begin{proof} We may assume without loss of generality that $F(0,x)$ is separable. Otherwise find an $\al\in\k$ such that $F(\al,x)$ is separable and replace
$t$ with $t-\al$ (such an $\al$ exists because $\Disc_xF(t,x)\neq 0$ since $F$ is separable over $\k(t)$). This does not change the discriminant of $F(t,\f)$. By Lemma \ref{lemsep} $F(0,\f)$ is separable over $k(\a)$.

Let us assign weights to the variables $a_j$ by $w(a_j)=j$ (and to monomials by additivity). For a polynomial $H\in\k[\a]$ we will denote by 
$\deg_wH$ the highest weight of a monomial appearing in it. The fact that $F$ is monic in $x$ and $n\ge\sl F$ implies that $N=\deg_t F(t,\f)=rn$. Write
$$F(t,\f)=\sum_{j=0}^N C_j(\a)t^j,C_j\in\k[\a],$$
$$F(0,\f)=\sum_{j=0}^N D_j(\a)t^j,D_j\in\k[\a].$$ 
We have $\deg_w C_j\le j$. Moreover, the degree $j$ form of each $C_j$ w.r.t. $w$ is exactly $D_j$ (since a polynomial of the form $t^\mu\f^\nu$ has a coefficient of weight $j-\nu$ at $t^j$), which is homogeneous of degree $j$ by construction. We have
$\deg_t F(t,\f)=\deg_t F(0,\f)=rn=N$. It follows from the homogeneity of the discriminant with weight $j$ for the coefficient of $t^j$ and the fact that $\Disc_t F(0,\f)\neq 0$ that the degree $N(N-1)$ form of $\Disc_tF(t,\f)$ is exactly $\Disc_t F(0,\f)$, which is homogeneous of degree $N(N-1)$.
This is the leading (highest weight) form of $\Disc_tF(t,\f)$. It is therefore enough to show that $\Disc_t F(0,\f)$ is not a square.
But this is just a special case of the proposition for a polynomial with constant coefficients (i.e. independent of $t$) and this has been proved in \cite[Proposition 1.7]{constant} for odd $q$.
\end{proof}

\subsection{Sign independence: $p=2$}

Assume $p=2$. We recall the definition and basic facts about the Berlekamp discriminant. See \cite{berlekamp} and \cite{carmon} for more details. Let $N$ be a natural number and $b_0,\ldots,b_N$ free variables. In the case of even characteristic the discriminant $\D_N(b_0,\ldots,b_N)\in\F_2[b_0\ldots,\b_N]$ (reduced modulo 2) is in fact the square of a polynomial $\del_N(b_0,\ldots,b_N)\in\F_2[b_0,\ldots,b_N]$. For a field $K\supset\F_2$ and a polynomial $g=\sum_{j=0}^NB_jt^j\in K[t],B_N\neq 0$ we will denote $\del(g)=\del_N(B_0,\ldots,B_N)$. If $\rho_1,\ldots,\rho_N$ are the roots of $g$ in an algebraic closure of $K$ the \emph{Berlekamp discriminant} of $g$ is defined to be
$$\BDisc(g)=\sum_{1\le i<j\le N}\frac{\rho_i\rho_j}{\rho_i^2+\rho_j^2}.$$ It can be written as $$\BDisc(g)=\frac{\xi(B_0,\ldots,B_N)}{\del(g)^2},$$ where $\xi(b_0\ldots,b_N)\in\F_2(b_0,\ldots,b_N)$ is a universal polynomial depending on $N$. We will denote $\xi(g)=\xi_N(g)$ when $\deg g=N$. If we assign the weights $w(b_i)=i$ to the variables $b_0,\ldots,b_N$ then $\del_N$ is homogeneous of degree N(N-1)/2 and $\xi_N$ is homogeneous of degree $N(N-1)$.

A fundamental property of the Berlekamp discriminant is that the Galois group of $g$ over $K$ contains an odd permutation of the roots of $g$ iff there exists a $\tau\in K$ such that $\BDisc(g)=\tau^2+\tau$. We now need to prove the following analogue of Proposition 5.1 for even characteristic:

\begin{prop}\label{52} Let $F\in k[t][x]$ with $\deg_x F=r>0$ be separable over $k(t)$ and monic in $x$. Assume that $n\ge\max(3,\sl F)$ and as usual $\a=(a_0,\ldots,a_n)$ are free variables, $\f=\sum_{j=0}^na_jt^j$. Then $\BDisc_t F(t,\f)$ is not of the form $\tau^2+\tau$ for $\tau\in k(\a)$.\end{prop}

\begin{proof} Assume to the contrary that there exists $\tau\in k(\a)$ such that $$\BDisc_t F(t,\f)=\tau^2+\tau.$$ Since $\BDisc_t F(t,\f)=\xi(F(t,\f))/\del(F(t,\f))^2$ and $k(\a)$ is a unique factorisation domain, we can write $\tau=u/\del(F(t,\f)),u\in k(\a)$, and we have \beq\label{4}\xi(F(t,\f))=u^2+\del(F(t,\f))u.\eeq

Now let us assign the weights $w(a_i)=i$ to the variables and denote by $\lf H$ the leading form of a polynomial $H\in k(\a)$ with respect to this weight. By the homogeneity properties of $\del_N,\xi_N$ and the fact that $$\deg F(t,\f)=\deg F(0,\f)=n\deg_x F$$ (since $n\ge\sl F$) we have that (denoting $N=n\deg_x F$)
$$\lf\del(F(t,\f))=\del(F(0,\f)),\lf\xi(F(t,\f))=\xi(F(0,\f)),$$ $$\deg_w\del(F(t,\f))=\deg_w\del(F(0,\f))=N(N-1)/2,$$ $$\deg_w\xi(F(t,\f))=\deg_w\xi(F(0,\f))=N(N-1),$$ and from \rf{4} we also
have $\deg_wu=N(N-1)/2$. Using these facts and taking leading forms in \rf{4} we deduce that $$\xi(F(0,\f))=(\lf u)^2+\del(F(0,\f))\cdot\lf u,$$
so taking $\tau_1=\lf u/\del(F(0,\f))$ we have $$\BDisc(F(0,\f))=\tau_1^2+\tau_1.$$
Now assuming, as we may after a shift in the variable $t$, that $F(0,\f)\in k(\a)[t]$ is separable, we have reduced our problem to the constant coefficient case (i.e. the case when $F(t,x)$ is independent of $t$). But this is a special case of \cite[Lemma 6.3]{carmon}.\end{proof}

\section{Proof of Theorem \ref{thmnm}}\label{secnm}

The proof of Theorem \ref{maingalois}, from which Theorem \ref{thmnm} follows, will occupy the present section as well as the next one.
Let $p$ be a prime number. Denote $\k=\fpb$. Let $F_1,\ldots,F_m\in\k[t][x],\deg_x F_i>0$ be non-associate, irreducible and separable over $\k(t)$. Denote
$F=\prod_{i=1}^mF_i$. Let $n\ge 3$ be a natural number. At this point we impose no further restrictions on the $F_i,n$ or $p$. They will be required later. Let $\a=(a_0,\ldots,a_n)$ be free variables over $\k$,
$$\f=\sum_{j=0}^na_it^i\in\k[\a][t], N_i=\deg_t F_i(t,\f).$$ By Lemma \ref{lemsep} the polynomial $F(t,\f)\in\k[\a][t]$ is separable over $\k(\a)$. Let $L$ be the splitting field
of $F(t,\f)$ over $\k(\a)$, $G=\Gal(L/\k(\a))$ its Galois group. We view $G$ as a subgroup of $S_{N_1}\times...\times S_{N_m}$ via its action on the roots
of each $F_i(t,\f)$. Our aim is to show that in fact $G=S_{N_1}\times...\times S_{N_m}$ under the assumptions of Theorem \ref{maingalois}.
Denote by $G_i\ss S_{N_i}$ the Galois group of $F_i(t,\f)$ over $\k(\a)$. This is the projection of $G$ to $S_{N_i}$ defined by restricting its action
to the roots of $F_i(t,\f)$.

By Proposition \ref{trans} each $G_i$ acts $2$-transitively (in fact $(n+1)$-transitively) on the roots of $F_i(t,\f)$.
Suppose that we could show that for each $i=1,\ldots,m$ there exists an element $\sig\in G$ which transposes two roots of $F_i(t,\f)$ and leaves all the other roots
of $F(t,\f)$ fixed. Then $G_i=S_{N_i}$, since $G_i$ is 2-transitive and $S_{N_i}$ is generated by transpositions. Furthermore the sign projection map
$G\to\prod_{i=1}^m S_{N_i}/A_{N_i}\cong\{\pm 1\}^m$ is onto. Lemma \ref{lem1} would then imply that $G=\prod_{i=1}^mS_{N_i}$.

It is therefore sufficient to prove the existence of transpositions as above. Our plan is to construct a discrete valuation ring in $k(\a)$ which ramifies in $L$ such that its inertia group contains the required transposition. An important ingredient in the proof is the following technical claim, the main idea behind its proof suggested to the author by U. Zannier.

\begin{prop}\label{keyprop} Let $H\in \k[\a]\sm\k[a_n]$ be an irreducible polynomial. Let $K$ be the algebraic closure of the field of fractions of
$\k[\a]/H$. Let $b_i$ be the image of $a_i$ in $K$, $g=\sum_{j=0}^m b_it^i\in K[t]$. Then one of the following holds for $F(t,g)\in K[t]$:
\begin{enumerate}
\item $F(t,g)$ is separable, i.e. has only simple roots in $K$.
\item $F(t,g)$ has one root of multiplicity two and the other roots are simple.
\item There exists a point $(\tau,\xi)$ on the affine plane curve $C$ defined by $F(t,x)=0$ with multiplicity $m_P$ such that $H\sim \f(\tau)-\xi\in\k[\a]$ $($by $\sim$ we denote association$)$, $\tau$ is a root of $F(t,g)$ of exact multiplicity $m_P$ and the other roots of $F(t,g)$ are simple.
\end{enumerate}
\end{prop} 

The proof of Proposition \ref{keyprop} will be given in Section \ref{seckeyprop}.
We now proceed to the proof of Theorem \ref{maingalois}.
Assume that $$n>\#\{P\in C|m_P>2\}.$$ Also assume that $p>\max N_i$. We have seen that it is enough to produce for each $1\le i\le m$ an element $\sig\in G$ which transposes two roots of $F_i(t,\f)$ and fixes the other roots of $F(t,\f)$. By symmetry it is enough to show this for $i=1$. Let $(\tau_l,\xi_l),l=0,\ldots,n-2$ be distinct points on $C$ including all the points $P$
such that $m_P>2$. Denote $$a_l'=\f(\tau_l)-\xi_l,$$ $$\K=\k(a_0',\ldots,a_{n-2}',a_n),$$ $\al$ a root of $F_1(t,\f)$ in $L$ (recall that $L$ is the splitting field
of $F(t,\f)$ over $\k(\a)$). The variables $a_0',\ldots,a_{n-2}',a_{n-1},a_n$ are obtained from $a_0,\ldots,a_n$ by an invertible affine transformation (it is invertible because of the nonvanishing of Vandermonde determinants with distinct second column entries). We have
$\k(\a)=\K(a_{n-1})$. The field $\k(\a,\al)=\K(a_{n-1},\al)$ can now be viewed as a one-variable function field over $\K$.

\begin{lem} The polynomial $F_1(t,\f)$ considered as an element of $\K[t,a_{n-1}]$ is irreducible over $\Kb$.\end{lem}

\begin{proof} We may write \beq\label{fuv}\f=u(t)+a_{n-1}v(t),\deg_tu,\deg_tv\le n,\eeq where $u\in\K[t]$ has coefficients which are affine forms in $a_0',\ldots,a_{n-2}',a_n$ and $v(t)\in\k[t]$. We have $v\neq 0$ again by the nonvanishing of the Vandermonde determinant with second column entries $\tau_0,\ldots,\tau_{n-2},t$. By \rf{fuv} we have $$a_{n-1}=(\f-u(t))/v(t).$$ The elements $t,\f\in\K(a_{n-1},t)$ are algebraically independent over $\Kb$. Suppose that $F_1(t,\f)=U(t,a_{n-1})V(t,a_{n-1})$, where $U,V\in\Kb[t,a_{n-1}]$ are nonconstant. Then
$$F_1(t,\f)=U\lb t,\frac{\f-u(t)}{v(t)}\rb V\lb t,\frac{\f-u(t)}{v(t)}\rb.$$ But this is impossible since $F_1(t,\f)$ is irreducible in $\Kb[t,\f]$ and by
Gauss's Lemma also in $\Kb(t)[\f]$. We obtained a contradiction.\end{proof}

\begin{prop}\label{ram} There exists an irreducible polynomial $H\in\k[\a]$ which when viewed as a polynomial in $\K[a_{n-1}]$ is nonconstant and defines a place which is ramified in the extension
$$\k(\a)=\K(a_{n-1})\ss \K(a_{n-1},\al)=\k(\a,\al)$$ of one-variable function fields over $\K$ (recall that $\al$ is a root of $F_1(t,\f)$).\end{prop}

\begin{proof} By the previous lemma the extension $\K(a_{n-1})\ss \K(a_{n-1},\al)$ is geometric, i.e. $\K(a_{n-1})$ is algebraically closed in $\K(a_{n-1},\al)$.
We assumed that $p>N_1$, so the Galois closure of $\K(a_{n-1},\al)$ over $\K(a_{n-1})$ is a Galois extension of degree prime to $p$ and so must ramify at some finite place
of $\K(a_{n-1})$ (the tame fundamental group of the affine line is zero, see \cite[\S XIII, Corollary 2.12]{sga}), which can be defined by some irreducible polynomial $H\in\k[\a]=\k[a_0',\ldots,a_{n-2}',a_n][a_{n-1}]$. Of course $H$ ramifies in $\K(a_{n-1},\al)$ as well. This proves the proposition.\end{proof}

We remark that the proof of the Proposition \ref{ram} is the only place where we use the condition $p>\max N_i$. Now let $H$ be as asserted in the proposition. Note that since $H$ is nonconstant as a polynomial in $\K[a_{n-1}]$ we have $H\not\in\k[a_n]$ and also $H$ is not associate to a polynomial of the form $f(\tau)-\xi$ for any point $P=(\tau,\xi)\in C$ with $m_P>2$ (since these elements are in $\K$ by construction). The image of the polynomial
$F(t,\f)\in\k[\a][t]$ modulo $H$ has degree $N=\deg_tF(t,\f)$ (since the leading coefficient of $F(t,\f)$ is in $k[a_n]$ and so is prime to $H$) and by
Proposition \ref{keyprop} it has at most one double root over the algebraic closure of the fraction field of $\k[\a]/H$, the other roots being simple, and no root of multiplicity 3 or higher. 

Denote by $R$ the discrete valuation ring $\K[a_{n-1}]_H$ in $\K(a_{n-1})$ and by $S$ any discrete valuation ring lying over
it in $L$. The field $R/HR$ is isomorphic to the field of fractions of $\k[\a]/H$ as a $k[\a]$-module (recall that $\k[\a]=\K[a_{n-1}]$). Denote by $\eta$ a prime element of $S$.
Since $H$ is ramified in the extension $\K[a_{n-1}]=\k[\a]\ss L$, the inertia group of $S$ relative to $R$ is non-empty, so there exists a nontrivial $\sig\in G$
satisfying $\sig\al_j\equiv\al_j\pmod{\eta S}$ for any root $\al_j$ of $F(t,\f)$. By the previous paragraph there can be at most two roots $\al_1,\al_2$ of $F(t,\f)$ for which
$\al_1\equiv\al_2\pmod{\eta S}$ and they must be roots of $F_1(t,\f)$ since $R$ ramifies in the extension defined by $F_1(t,\f)$. Therefore $\sig$ transposes
$\al_1,\al_2$ and leaves the other roots of $F(t,\f)$ fixed, which is exactly what we needed to complete the proof of Theorem \ref{maingalois}.

\section{Proof of Proposition \ref{keyprop}}\label{seckeyprop}

We keep the setting and notation of the previous section, but we assume no restrictions on $n$ and $p$ other than $n\ge 3$.
Let $H\in \k[\a]\sm\k[a_n]$ be an irreducible polynomial. Let $K$ be the algebraic closure of the field of fractions of
$\k[\a]/H$. Let $b_i$ be the image of $a_i$ in $K$, $g=\sum_{j=0}^m b_it^i\in K[t]$.

We begin by noting that $H(b_0,\ldots,b_n)=0$ and up to a constant this is the only relation satisfied by $b_0,\ldots,b_n$. 
By our assumption that $H\not\in\k[a_n]$ we have $$\deg_t F_i(t,g)=N_i=\deg_t F_i(t,\f),$$ since the leading coefficient of $F_i(t,f)$ is in $\k[a_n]$ and so is prime to $H$.

Denote $d=\deg_x F$. Over a finite separable extension $E$ of $\k(t)$ we may factor 
$$F(t,x)=c(t)\prod_{i=1}^d (x-\ze_i/c(t)),$$
where $\ze_i\in E$ are integral over $\k[t]$ and $c(t)$ is the leading coefficient of $F$ as a polynomial in $x$. The $\ze_i$ are distinct since the $F_i$ are
distinct, irreducible and separable. We will denote $\phi_i=\xi_i/c(t)\in E$.

A \emph{place} on a field $\mathcal{F}$ with values in a field $\mathcal{E}$ is a map $\pi:\mathcal{F}\to\mathcal{E}\cup\{\ity\}$ such that
$\mathcal{R}=\pi^{-1}(\mathcal{E})$ is a valuation ring in $\mathcal{F}$ and $\pi|_\mathcal{R}$ is a ring homomorphism. See \cite[\S 9.7,9.8]{jacobson} for the definition and basic properties of valuation rings and places. The most important fact we will use is that if $\mathcal{E}$ is algebraically closed and $R\ss\mathcal{F}$ is any subring, then any homomorphism $\pi:R\to\mathcal{F}$ can be extended to an $\mathcal{E}$-valued place on $\mathcal{F}$. 
If $\pi$ is a place on $\mathcal{F}$ which is regular on $R$, i.e. does not assume $\ity$, then it is also regular on the integral closure of $R$ in $\mathcal{K}$.

Recall that $K$ is the algebraic closure of the field of fractions of $\k[\a]/H$. Let $\al\in K$ be any element. Evaluation at $\al$ defines a place $\pi_\al:\k(t)\to K\cup\{\ity\}$ which is regular (i.e. does not assume $\ity$) on $k[t]$.
For each $\al\in K$, $\pi_\al$ can be
extended to a place $E\to K\cup\{\ity\}$ which we will also denote by $\pi_\al$. We choose one such extension for each $\al\in K$. 
Since the $\ze_i$ are integral over $\k[t]$ we have $\pi_\al(\xi_i)\in K$. If $c(\al)\neq 0$ then $\pi_\al(\phi_i)\in K$.
For $h\in E$ we will use the notation $h(\al)=\pi_\al(h)$. For $h\in k(t)$ this coincides with the usual meaning of $h(\al)$.
Note that for $h\in E$ and any algebraically closed field $\k\ss K'\ss K$ we have $h(\al)\in K'\cup\{\ity\}$ whenever $\al\in K'$. This is because $h$ is
algebraic over $\k(t)$, so $h(\al)$ is algebraic over $\k(\al)$ (since $\pi_\al$ is a ring homomorphism on its valuation ring).

The usual derivative defines a derivation $D:\k(t)\to \k(t)$ over $\k$. Since $E/\k(t)$ is a finite separable extension, $D$ extends uniquely to a derivation
$D:E\to E$. For $h\in E$ we will denote $h'=Dh$. On $\k(t)$ this coincides with the usual definition of the derivative.
Similarly the usual derivative on $K(t)$ can be extended to $K(t)E=K(t)[x]/F(t,x)$ and this extension is coherent with the extension from $\k(t)$ to $E$ by uniqueness.

While we may assume $p>2$ for our application, we will prove Proposition \ref{keyprop} for $p=2$ as well, as this only requires a slight modification and the proposition might be useful in full generality. To accomodate the case $p=2$ we will need to use the second Hasse-Schmidt derivative (see \cite[\S 1.3]{goldschmidt} for background on this notion). For a polynomial $f=\sum_{j=0}^e u_jt^j\in\ell[t]$ over a field $\ell$ it is defined by
$$f^{(2)}=\sum_{j=2}^e u_j\lb\begin{array}{c}j\\2\end{array}\rb t^{j-2}.$$
The second Hasse-Schmidt derivative is $\ell$-linear and satisfies
\begin{equation}\label{prodrule}(f_1f_2)^{(2)}=f_1^{(2)}f_2+f_1'f_2'+f_1f_2^{(2)}.\end{equation} In characteristic $\neq 2$ we have $f^{(2)}=\frac{1}{2}f''$.
An element $\al\in\bar{\ell}$ is a triple (or higher multiplicity) root of $f$ iff $f(\al)=f'(\al)=f^{(2)}(\al)=0$.
Like the usual derivative, the second Hasse-Schmidt derivative has a unique extension to $\ell(t)$ and then to any finite separable extension of it.
We extend the second Hasse-Schmidt derivative from $\k(t)$ to $E$ and from $K(t)$ to $K(t)E$ (these are coherent extensions by uniqueness, i.e. the derivative on $K(t)E$ when restricted to $E$ coincides with the derivative on $E$).

After setting up the notions of evaluation of elements of $E$ (which we view as algebraic functions on $K$) on elements of $K$ and the notions of derivative
and second Hasse-Schmidt derivative for elements (algebraic functions) in $K(t)E$, we proceed to the proof of Proposition \ref{keyprop}.

First assume that there exists an element $\tau\in\k$ such that $F(\tau,g(\tau))=0$. Since $\k$ is algebraically closed we have $\xi=\g(\tau)\in\k$.
The point $(\tau,\xi)\in\mathbf{A}^2(k)$ lies on the curve $C$ defined by $F(t,x)=0$. We have $g(\tau)=\sum_{j=0}^nb_j\tau^j=\xi$, so 
$$H\sim \sum_{j=0}^na_j\tau^j-\xi$$ (this relation is irreducible since it is linear and so it is associate with $H$ by uniqueness). 
We also see that $\tau$ is the only root of $F(t,g)$ contained in $\k$ (again by the uniqueness of the algebraic relation satisfied by $b_0,\ldots,b_n$).
For simplicity we assume that $\tau=\xi=0$, otherwise we may shift the variables
$t,x$ by a constant without affecting either the assumptions or the conclusion of the proposition. Then $H\sim a_0, b_0=0$ and $b_1,\ldots,b_n$ are free variables
over $\k$. We also have $g=b_1t+\ldots+b_nt^n$.

Let $m_P$ be the multiplicity of $(\tau,\xi)=(0,0)$ as a point on $C$. We claim that
the multiplicity of $0$ as a root of $F(t,g)\in K[t]$ is exactly $m_P$. 
Let $F=\sum_{l=m_P}^d\FF_l(t,x),\deg\FF_l=l$ be the decomposition of $F$ into homogeneous forms.
We have $$F(t,g)=\sum_{l=m_P}^d\FF_l(t,b_1t+\ldots+b_nt^n)=\FF_{m_P}(1,b_1)t^{m_P}+\mbox{terms of degree}>m_P.$$
Since $\FF_{m_P}\neq 0$ this proves our claim.

Now we want to show that any other root $\al\neq 0$ of $F(t,g)$ is simple. Let $\al\neq 0$ be such a root. We have observed that necessarily $\al\not\in\k$.
In particular $c(\al)\neq 0$ (since $\k$ is algebraically closed and $c(t)\in\k(t)$). Therefore $\phi_i=\ze_i/c$ are regular at $\al$ (i.e. $\pi_\al$ is regular
at $\phi_i$) and so are $\phi_i',\phi_i^{(2)}$.
We have $$F(\al,g(\al))=c(\al)\prod_{i=1}^d\lb g(\al)-\phi_i(\al)\rb=0,$$ so for some $i$ we must have $g(\al)=\phi_i(\al)$,
i.e. \beq\label{geq1}b_1\al+b_2\al^2+\ldots+b_n\al^n=\phi_i(\al)\eeq We will assume that $i=1$, so $g(\al)=\phi_1(\al)$. Now assume that $\al$ is a double root
of $F(t,g)$. Then $F(t,g)'(\al)=0$. We have
$$F(t,g)'=c(t)\sum_{i=1}^d(g'-\phi_i')\prod_{1\le j\le d\atop{j\neq i}}(g-\phi_j)+c'(t)\prod_{i=1}^d(g-\phi_i),$$
therefore either $g'(\al)=\phi_1'(\al)$, or $g(\al)=\phi_j(\al)$ for some $j\neq 1$. The latter cannot happen since then $\phi_j(\al)=\phi_1(\al)$ and so
$\al\in\k$ (because $(\ze_i-\ze_1)(\al)=0$, $\ze_i-\ze_1$ divides some nonzero polynomial in $\k[t]$ and $\k$ is algebraically closed), which is a contradiction.
So we have \beq\label{geq2}g'(\al)=b_1+2b_2\al+\ldots+nb_n\al^{n-1}=\phi_1'(\al).\eeq Multiplying \rf{geq2} by $\al$ and subtracting from \rf{geq1} we obtain
$$-b_2\al^2-2b_3\al^3-...-(n-1)b_n\al^n=\phi_1(\al)-\al\phi_1'(\al).$$
Since $\al\neq 0$ it follows that $b_2$ lies in the algebraic closure of $\k(b_3,\ldots,b_n,\al)$ and by \rf{geq2} so does $b_1$. This implies that
the transcendence degree of $\k(b_1,\ldots,b_n)$ over $\k(b_3,\ldots,b_n)$ is at most one, which is a contradiction since $b_1,\ldots,b_n$ are algebraically independent
over $\k$.

Now we handle the case when $F(t,g)$ has no roots in $\k$. First we show that $F(t,g)$ has no root of multiplicity 3 or higher. Assume to the contrary that 
$\al\in K$ is such a root. Then $F(t,g)(\al)=F(t,g)'(\al)=F(t,g)^{(2)}(\al)=0.$ As above this implies (using the product rule \rf{prodrule} and the fact that $\al\not\in\k$) that for some $i$ we have 
$$g(\al)=\phi_i(\al),g'(\al)=\phi_i'(\al),g^{(2)}(\al)=\phi_i^{(2)}(\al).$$
We assume that this happens for $i=1$. The relation
$$g^{(2)}(\al)=b_2+3b_3\al+\ldots+\lb\begin{array}{c}n\\2\end{array}\rb\al^{n-2}=\phi_1^{(2)}(\al)$$ implies that $b_2$ is algebraic over $\k(b_3,\ldots,b_n,\al)$.
The same then follows for $b_1,b_0$ from the relations $g'(\al)=\phi_1'(\al),g(\al)=\phi_1(\al)$. This implies that $\k(b_0,\ldots,b_n)$ has transcendence
degree 1 over $\k(b_3,\ldots,b_n,\al)$, so the transcendence degree of $\k(b_0,\ldots,b_n)$ over $\k$ is at most $n-1$. This is a contradiction since $b_0,\ldots,b_n$ satisfy only one algebraic relation over $\k$.

Finally we want to exclude the possibility of two double roots $\al,\be\not\in\k, \al\neq\be$ of $F(t,f)$. Assume to the contrary that such $\al,\be$ exist.
Arguing as in the previous cases we see that there must exist $i\neq j$ such that 
\beq\label{eq4}g(\al)=\phi_i(\al),g'(\al)=\phi_i'(\al),g(\be)=\phi_j(\be),g'(\be)=\phi_j'(\be).\eeq
We assume $i=1,j=2$. The relations \rf{eq4} imply that
$$b_0+b_1\al+b_2\al^2+b_3\al^3,b_0+b_1\be+b_2\be^2+b_3\be^3,$$
$$b_1+2b_2\al+3b_3\al^2,b_1+2b_2\be+3b_3\be^2$$ are all algebraic over $\k(b_4,\ldots,b_n,\al,\be)$ (recall that $n\ge 3$).
This gives an inhomogeneous linear system of equations for $b_0,b_1,b_2,b_3$ over the algebraic closure of $\k(b_4,\ldots,b_n,\al,\be)$ with determinant
$$\det\lbb\begin{array}{cccc} 1 & \al & \al^2 & \al^3 \\ 1 & \be & \be^2 & \be^3 \\ & 1 & 2\al & 3\al^2 \\ & 1 & 2\be & 3\be^2\end{array}\rbb=-(\al-\be)^4\neq 0.$$
This implies that $b_0,b_1,b_2,b_3$ are algebraic over $\k(b_4,\ldots,b_n,\al,\be)$, which is a contradiction since $b_0,\ldots,b_n$ satisfy only one algebraic relation
over $\k$. This concludes the proof of Proposition \ref{keyprop}.

{\bf Acknowledgments.} The author would like to thank Ze\'{e}v Rudnick for many useful discussions and for his encouragement during the research leading to this paper. The present work is part of the author's Ph.D. studies at Tel-Aviv University under his supervision.
The author would also like to thank Umberto Zannier for suggesting the main idea behind the proof of Proposition \ref{keyprop} and Dan Carmon for suggesting the way to treat the even characteristic case.
The author would also like to thank Lior Bary-Soroker for some useful discussions, Brian Conrad for some useful remarks and suggestions and Keith Conrad for his careful reading of the paper and many helpful remarks about the exposition. The author would like to thank the MathOverflow community and especially Peter M\"{u}ller for pointing out some of the results on multiply transitive groups used in the present work. Finally the author would like to thank the anonymous referee of this paper for providing multiple expositional corrections and suggestions.

\end{document}